\documentclass[reqno, 11pt]{amsart}
\usepackage{amsfonts}
\usepackage{amsmath}
\usepackage{tikz-cd}
\usepackage{mathtools}
\usepackage{amssymb}
\usepackage{amsthm}
\usepackage{amscd}
\usepackage{color}
\usepackage{verbatim}

\usepackage[colorlinks]{hyperref}
\definecolor{linkblue}{RGB}{1,1,190}
\definecolor{citered}{RGB}{190,1,1}  
\hypersetup{
	linkcolor=linkblue,
	urlcolor=linkblue,
	citecolor=citered
}

\newtheorem{theorem}{Theorem}
\newtheorem{lemma}[theorem]{Lemma}

\newtheorem{proposition}[theorem]{Proposition}
\newtheorem{corollary}[theorem]{Corollary}

\theoremstyle{definition}
\newtheorem{example}[theorem]{Example}

\newcommand{\N}{\mathbb N}
\newcommand{\Z}{\mathbb Z}

 \DeclareMathOperator{\ord}{ord}
 
\DeclareMathOperator{\fin}{fin} \DeclareMathOperator{\rev}{rev}
 
\def\P{\ensuremath\mathcal{P}}
\def\Pf{\ensuremath\mathcal{P}_{\fin,1}}

\def\O{\ensuremath\mathcal{O}}

\title[The isomorphism problem for reduced finitary power monoids]{The isomorphism problem for reduced \\ finitary power monoids}
\author{Balint Rago}
\address{University of Graz, NAWI Graz, Department of Mathematics and Scientific Computing, Heinrichstraße 36,
8010 Graz, Austria}

\thanks{This work was supported by the FWF (projects W1230 and  10.55776/DOC-183-N)}

\email{balint.rago@uni-graz.at}

\subjclass[2020]{20M14}

\keywords{isomorphism problem, power monoid}

\begin{document}

\maketitle

\begin{abstract}
    Let $H$ be a multiplicatively written monoid with identity $1_H$ and let $\Pf(H)$ denote the reduced finitary power monoid of $H$, that is, the monoid consisting of all finite subsets of $H$ containing $1_H$ with set multiplication as operation. Building on work of Tringali and Yan, we give a complete description of pairs of commutative and cancellative monoids $H,K$ for which $\Pf(H)$ and $\Pf(K)$ are isomorphic.
\end{abstract}

\section{Introduction}

Let $S$ be a multiplicatively written semigroup. We denote by $\P(S)$ the \textit{large power semigroup} of $S$, i.e.\ the family of all non-empty subsets of $S$, endowed with the binary operation of setwise multiplication \[(X,Y)\mapsto \{x\cdot y:x\in X,y\in Y\}.\] Moreover, for a monoid $H$ with identity $1_H$, we denote by $\Pf(H)$ the set of finite subsets of $H$ that contain the identity element of $H$; it is a submonoid of $\P(H)$ with identity $\{1_H\}$, called the \textit{reduced finitary power monoid} of $H$. 

The systematic study of power semigroups began in the 1960s and was initiated by Tamura and Shafer.\ A central question that arose from their work, called the \emph{isomorphism problem for power semigroups}, is whether, for semigroups $S$ and $T$ in a certain class $\O$, an isomorphism between $\P(S)$ and $\P(T)$ implies that $S$ and $T$ are isomorphic. Although this was answered in the negative by Mogiljanskaja \cite{Mo73} for the class of all semigroups, other classes have been found for which the answer is positive, such as the class of groups \cite{Sh67}, Clifford semigroups \cite{Ga-Zh14} and several others \cite{Go-Is84,Sh-Ta67,Tr25}. Recently, the investigation of the arithmetic of reduced finitary power monoids was set in motion \cite{An-Tr21,Bi-Ge25,Co-Tr25,Fa-Tr18,Tr22,Re26a}, motivated by classical questions in additive number theory, such as Sárközy's conjecture \cite[Conjecture 1.6]{Sa12} and Ostmann's conjecture \cite[p.13]{Os68}. 
For recent work on the structure of automorphisms of power monoids, see \cite{2Ra25,Tr-We25,Tr-Ya25}. \\

In \cite{Bi-Ge25}, the authors conjectured that the isomorphism problem for reduced finitary power monoids has a positive answer for the class of numerical monoids. In other words, whenever $H$ and $K$ are (additive) submonoids of $\N_0$, then $\P_{\fin,0}(H)$ is isomorphic to $\P_{\fin,0}(K)$ if and only if $H$ and $K$ are isomorphic. It is straightforward to verify that for any two monoids $H$ and $K$ and an isomorphism \[f:H\to K,\] the map \[F:\Pf(H)\to\Pf(K),\]  \[F(X)=\{f(x):x\in X\}\] itself is an isomorphism, whence only the ``only if" part of the statement is of interest.\\

The aforementioned conjecture was soon resolved by Tringali and Yan \cite{2Tr-Ya25}, who proved that the isomorphism problem has a positive answer for the class of rational Puiseux monoids, i.e. the class of additive submonoids of $\mathbb{Q}_{\geq 0}$. Moreover, they provided several examples of non-isomorphic monoids with isomorphic reduced finitary power monoids \cite[Examples 1.2]{2Tr-Ya25}. While these examples all involve non-cancellative monoids, another such example was given in \cite{Ra25}, giving a ``strongly negative" answer to the isomorphism problem for the class of commutative valuation monoids. 
To be more precise, it is shown that for two non-isomorphic reduced valuation monoids $H$ and $K$ with isomorphic quotient groups, there is always an isomorphism $\Pf(H)\to\Pf(K)$. Further progress has been achieved in \cite{3Tr-Ya25}, where an affirmative answer is given for the class of torsion groups. This raises the question, in what way power monoids of valuation monoids have more flexibility than those of other ``well-behaved" classes of monoids.  \\

In this paper, we give a complete classification of the isomorphism problem for reduced finitary power monoids for the class of commutative, cancellative monoids. This is achieved in two steps. First, we provide a positive answer to the isomorphism problem for the class of commutative, cancellative monoids with non-trivial unit group (Theorem \ref{nonreducedpositive}). Thereafter, we investigate reduced monoids and generalize the result obtained in \cite{Ra25} to non-valuation monoids. In equivalent terms, our main result (Theorem \ref{main}) states that for two commutative, cancellative non-isomorphic monoids $H$ and $K$, $\Pf(H)\simeq\Pf(K)$ if and only if $H$ and $K$ are both reduced and $H$ differs from $K$ (up to isomorphism) only by a certain ``deformation" of a specifically defined valuation submonoid of $H$. Moreover, we provide an example of two such monoids, which are not valuation monoids. \\

In the next section, we gather some preliminaries and recall some results from \cite{3Tr-Ya25}. There, given two monoids $H$ and $K$ and an isomorphism $f:\Pf(H)\to\Pf(K)$, the authors prove the existence of a bijective map $g:H\to K$, which is closely related to $f$ and serves as a remarkably helpful tool to establish a strong relationship between the power monoids and their respective base monoids.

\section{Preliminaries and the pullback of an isomorphism}

We denote by $\N$ the set of positive integers and by $\N_0$ the set of non-negative integers. Let $G$ be an abelian group, $X\subseteq G$ a subset and $a\in G$. We write \[aX:=\{ax:x\in X\}\] and \[X^{-1}:=\{x^{-1}:x\in X\}.\]  

We say that a commutative monoid $H$ is \textit{cancellative} if for all $a,b,c\in H$ with $ac=bc$, it holds that $a=b$. For a commutative, cancellative monoid $H$, we denote by $\mathsf{q}(H)$ the \textit{quotient group} of $H$, that is, the unique abelian group up to isomorphism with the property that any abelian group containing an isomorphic image of $H$, also contains an isomorphic image of $\mathsf{q}(H)$. We have \[\mathsf{q}(H)=\{ab^{-1}:a,b\in H\}\] and throughout this paper, we will always assume that $H\subseteq \mathsf{q}(H)$.
We say that a commutative, cancellative monoid $H$ is a \textit{valuation monoid} if for every $x\in \mathsf{q}(H)$, we have $x\in H$ or $x^{-1}\in H$, or, equivalently, if for every $x,y\in H$, $xy^{-1}\in H$ or $yx^{-1}\in H$. 
Moreover, we denote by $H^\times$ the group of units (invertible elements) of $H$ and we call $H$ \textit{reduced} if $H^\times=\{1_H\}$. We call a non-invertible element $a\in H$ \textit{irreducible} if it cannot be written as a non-trivial product of non-invertible elements of $H$ and we denote by $\ord(a)$ the \textit{order} of an element $a\in H$, that is, the minimal positive integer $n$ such that $a^n=1_H$. If no such integer exists, then we say that $a$ has infinite order. Lastly, we say that two elements $a,b\in H$ are \textit{independent}, if for all $n,m\in\Z$, $a^n=b^m$ implies $n=0=m$, or equivalently, if the subgroup of $\mathsf{q}(H)$ generated by $a$ and $b$ is free abelian of rank 2. It is evident that two elements are independent only if they have infinite order.

Let $H$ be a monoid with identity $1_H$; $X,Y\in\Pf(H)$ and $n\in\mathbb{N}$. We denote by $X^n$ the $n$-fold product of $X$ in $\Pf(H)$, that is, \[X^n=\underbrace{X\cdot\ldots\cdot X}_{n \text{ times}}.\] We say that $X$ divides $Y$, or that $Y$ is divisible by $X$ in $\Pf(H)$ if $X\cdot Z=Y$ for some $Z\in\Pf(H)$. Note that, since $1_H\in Z$ for every $Z\in \Pf(H)$, $Y$ is divisible by $X$ only if $X\subseteq Y$.\\

Let $H$ and $K$ be monoids. In \cite{3Tr-Ya25}, Tringali and Yan proved that every isomorphism $f:\Pf(H)\to\Pf(K)$ exhibits a strong property, which we will use extensively in this paper.

\begin{theorem}[{{\cite[Theorem 3.2]{3Tr-Ya25}}}]
    Let $H,K$ be monoids and $f:\Pf(H)\to\Pf(K)$ an isomorphism. Then $f(X)$ is a 2-element set for every 2-element set $X\in\Pf(H)$.
\end{theorem}

By applying the theorem to the inverse of $f$, it follows that $f$ bijectively maps 2-element sets to 2-element sets. This allows us to define an auxiliary map $g:H\to K$ in the following way. We set $g(1_H)=1_K$ and for any $a\in H\setminus \{1_H\}$, $g(a)$ is defined via $f(\{1_H,a\})=\{1_K,g(a)\}$. This function is clearly a well-defined bijection and is called the \textit{pullback} of the automorphism $f$. We now present another result of Tringali and Yan, concerning a basic, yet crucial property of the pullback. 

\begin{proposition} [{{\cite[Proposition 4.3]{3Tr-Ya25}}}] \label{Tringaliprop}
    Let $H,K$ be cancellative monoids and $f:\Pf(H)\to\Pf(K)$ an isomorphism with pullback $g$. Then $g(a^n)=g(a)^n$ for all $a\in H$ and $n\in\N_0$.
\end{proposition}

As an immediate corollary of this proposition, we obtain that the pullback preserves the order of an element. Indeed, if $a\in H$, then $a^n=1_H$ if and only if $g(a^n)=g(1_H),$ which is equivalent to $g(a)^n=1_K$. Hence $\ord(a)=\ord(g(a))$ for all $a\in H$. We now add to the previous result by showing that the pullback preserves negative powers as well. 

\begin{lemma} \label{pullbackunit}
    Let $H,K$ be commutative and cancellative monoids and $f:\Pf(H)\to\Pf(K)$ an isomorphism with pullback $g$. If $a\in H^\times$, then $g(a)\in K^\times$ and $g(a^{-1})=g(a)^{-1}$. In particular, $g(a^n)=g(a)^n$ for all $n\in\Z$ and $g(H^\times)=K^\times$. 
\end{lemma}

\begin{proof}
    Let $a\in H^\times$ and set $x:=g(a)$ and $y:=g(a^{-1})$. If $a=1_H$, the assertion follows by definition of $g$, whence we can disregard this case. Moreover, if $a$ is an element of finite order $n$, then \[y=g(a^{n-1})=x^{n-1}=x^{-1}\] by Proposition \ref{Tringaliprop} and the observation that $\ord(a)=\ord(x)$. 
    Thus, we can assume that $a$ has infinite order. Note that the equation \[\{1_H,a\}^2\cdot \{1_H,a^{-1}\}=\{1_H,a,a^2,a^{-1}\}=\{1_H,a^2\}\cdot\{1_H,a^{-1}\}\] and Proposition \ref{Tringaliprop} imply that \[\{1_K,x\}^2\cdot \{1_K,y\}=\{1_K,x^2\}\cdot\{1_K,y\}.\] Then, since $\{1_K,x\}$ divides the right hand site of the equation, we obtain \[x\in \{1_K,x^2\}\cdot\{1_K,y\}=\{1_K,x^2,y,x^2y\}.\] By the bijectivity of $g$ and the fact that $1_H\neq a\neq a^{-1}$, both $x=1_K$ and $x=y$ are impossible. Moreover, since $a$ has infinite order, so does $x$, whence $x=x^2$ is ruled out as well. Consequently, $x=x^2y$ and $y=x^{-1}$, as desired. Hence $x$ is a unit of $K$, $g(H^\times)\subseteq K^\times$ and a symmetric argument involving $f^{-1}$ yields $g(H^\times)=K^\times$. Moreover, by Proposition \ref{Tringaliprop}, we obtain $g(a^n)=x^n$ for all $n\in \Z$.
\end{proof}

\section{Is the pullback an isomorphism?}

In this section, we give a positive answer to the isomorphism problem for reduced finitary power monoids for the class of commutative, cancellative monoids with non-trivial unit group (Theorem \ref{nonreducedpositive}). This is achieved via a straightforward approach: We prove that the pullback $g$ of an isomorphism $f:\Pf(H)\to \Pf(K)$ for two such monoids $H$ and $K$ is itself an isomorphism between $H$ and $K$, i.e., that $g(ab)=g(a)g(b)$ for all elements $a,b\in H$. 
To do this, we work in the general setting of commutative, cancellative monoids and we consider several different cases, depending on the order of $a$ and $b$ and whether they are independent elements or not. If, for example, $a$ and $b$ both have finite order, then it already follows from \cite[Proposition 4.7]{3Tr-Ya25} that the pullback preserves their product, whence we can turn our attention to monoids, which contain at least one element of infinite order. It will turn out that the class of reduced monoids is more complicated to handle. We will deal with them in the next section. \\

We further mention that whenever $a=b^k$ or $b=a^k$ for some $k\in\Z$, then it follows from Proposition \ref{Tringaliprop} and Lemma \ref{pullbackunit} that $g(ab)=g(a)g(b)$, whence we can disregard this case. Moreover, for the sake of readability, we will use the results obtained from Proposition \ref{Tringaliprop} and Lemma \ref{pullbackunit} freely and without further mention throughout the rest of this paper.

\begin{center}
    \textit{Throughout the rest of this paper, a monoid is always commutative, cancellative and contains at least one element of infinite order.}
\end{center}

Let $H$ be a monoid and let $X\in\Pf(H)$. We call an element $a\in H$, $a\neq 1_H$ a \textit{quotient} of $X$ if there exists $b\in X$ such that $ab\in X$. Moreover, we say that $a$ is a \textit{quotient of} $X$ \textit{of multiplicity} $n$, where $n$ is the number of elements $b\in X$ such that $ab\in X$. Clearly, since $1_H\in X$, every $a\in X\setminus \{1_H\}$ is itself a quotient of $X$. \\

We proceed with an important property of isomorphisms between power monoids.

\begin{lemma} \label{card}
    Let $H$ and $K$ be monoids. Then every isomorphism $f:\Pf(H)\to\Pf(K)$ is cardinality-preserving.
\end{lemma}

\begin{proof}
    Let $f:\Pf(H)\to\Pf(K)$ be an isomorphism with pullback $g$, $X\in\Pf(H)$, $a\in H$ an element of infinite order and $x:=g(a)$. Since $H$ is cancellative, we find $N\in\N$ such that $a^nX\cap X=\emptyset$ for every $n\geq N$. Moreover, for distinct $s,t\in \N_0$ and $n\geq N$, we have $a^{sn}X\cap a^{tn}X=\emptyset$. This implies that the equation \[\{1_H,a^n\}\cdot Y=\{1_H,a^n\}^3\cdot X\] in $\Pf(H)$ is satisfied if and only if $Y=X\cup a^{2n}X\cup Z$, where $Z$ is any subset of $a^nX$, whence the equation has precisely $2^{|X|}$ solutions in $\Pf(H)$ for almost all $n\in\N$. Then, the equation \[\{1_K,x^n\}\cdot Y=\{1_K,x^n\}^3\cdot f(X)\] has $2^{|X|}$ solutions in $\Pf(K)$ for almost all $n\in\N$. However, $x$ is an element of infinite order in $K$ and by repeating the argument above, we infer that the equation \[\{1_K,x^n\}\cdot Y=\{1_K,x^n\}^3\cdot f(X)\] has $2^{|f(X)|}$ solutions in $\Pf(K)$ for almost all $n\in\N$. Hence $|X|=|f(X)|$, which proves our assertion.

\end{proof}

As a consequence, we can immediately infer that quotients together with their multiplicities are preserved under the pullback of an isomorphism.

\begin{lemma} \label{quot}
    Let $H$ and $K$ be monoids, $f:\Pf(H)\to\Pf(K)$ an isomorphism with pullback $g$ and $X\in\Pf(H)$. Then $a\in H\setminus\{1_H\}$ is a quotient of $X$ of multiplicity $n$ if and only if $g(a)$ is a quotient of $f(X)$ of multiplicity $n$. 
\end{lemma}

\begin{proof}
    We observe that $a\in H\setminus\{1_H\}$ is a quotient of $X$ of multiplicity $n$ if and only if the set $\{1_H,a\}\cdot X$ contains precisely $2|X|-n$ elements. Since by Lemma \ref{card}, this is equivalent to the set $\{1_K,g(a)\}\cdot f(X)$ containing precisely $2|f(X)|-n$ elements, we are done.
\end{proof}

We begin with the investigation of how the pullback of an isomorphism $f:\Pf(H)\to\Pf(K)$ behaves with regards to products of elements. The following lemma will serve as a baseline.

\begin{lemma} \label{core}
    Let $H,K$ be monoids, $f:\Pf(H)\to \Pf(K)$ an isomorphism with pullback $g$ and $a\in H$ an element of infinite order.
    \begin{enumerate}
        \item If $b\in H$ with $\ord(b)=2$, then $g(ab)=g(a)g(b)$.
        \item If $b\in H$ with $\ord(b)\geq 3$, and $g(ab)\neq g(a)g(b)$, then $g(ab)\in\{g(a)g(b)^{-1},g(b)g(a)^{-1}\}$.
    \end{enumerate}

    \begin{proof} We set $x:=g(a), y:=g(b)$ and $z:=g(ab)$. \\

       \noindent (1) Note that the equation \[\{1_H,ab\}\cdot\{1_H,b\}=\{1_H,a,b,ab\}=\{1_H,a\}\cdot\{1_H,b\}\] implies that $\{1_K,z\}$ divides $\{1_K,x\}\cdot\{1_K,y\}$, whence $z\in \{1_K,x,y,xy\}$. By the bijectivity of $g$ and the fact that $ab\not\in \{1_H,a,b\}$, we then conclude that $z=xy$. \\

       \noindent (2) By the assumption that $z\neq xy$, we can assume that $a\neq b^k$ and $b\neq a^k$ for all $k\in\Z$, which clearly means that \[ab\not\in\{a,b,a^{-1},b^{-1}\}\] and \[z\not\in\{x,y,x^{-1},y^{-1}\}.\] We observe that $ab$ is a quotient of the set \[\{1_H,a\}\cdot\{1_H,b\}=\{1_H,a,b,ab\},\] whence by Lemma \ref{quot}, $z$ is a quotient of $\{1_K,x\}\cdot\{1_K,y\}$. Consequently, we have \[z\in \{x,y,x^{-1},y^{-1},xy,xy^{-1},yx^{-1},x^{-1}y^{-1}\}\] and, keeping our assumptions in mind, more specifically \[z\in \{xy^{-1},yx^{-1},x^{-1}y^{-1}\}.\] 
       Suppose now that $z=x^{-1}y^{-1}$, which implies that $x,y\in K^\times$ and consequently, that $a,b\in H^\times$. We set \[X:=\{1_K,x\}\cdot\{1_K,x^{-1}y^{-1}\}\cdot\{1_K,y\}=\{1_K,x,y,xy,x^{-1},y^{-1},x^{-1}y^{-1}\}\] and observe that the equation \[X=\{1_K,x^{-1}\}\cdot\{1_K,xy\}\cdot\{1_K,y^{-1}\}\] implies that both $a^{-1}$ and $b^{-1}$ are contained in \[f^{-1}(X)=\{1_H,a\}\cdot\{1_H,ab\}\cdot\{1_H,b\}=\{1_H,a,b,ab,a^2b,ab^2,a^2b^2\}.\] Using our initial assumptions, we conclude that $a^{-1}\in\{ab^2,a^2b^2\}$. However, if $a^{-2}=b^2$, then $x^{-2}=y^2$ and $xy=x^{-1}y^{-1}$. In particular, we have $z=xy$, contradicting our assumption. Hence we are left with $a^{-1}=a^2b^2$, or equivalently, $a^{-3}=b^2$ and by a symmetric argument, $b^{-3}=a^2$. However, these two equations clearly imply that $a=b$, another contradiction. Hence $z\in\{xy^{-1},yx^{-1}\}$, as desired.
    \end{proof}
\end{lemma}

We now continue with the case, where both $a,b\in H$ are elements of infinite order, which are not independent.

\begin{proposition}\label{relation}
    Let $H,K$ be monoids and let $f:\Pf(H)\to\Pf(K)$ be an isomorphism with pullback $g$. If $a,b\in H$ both have infinite order and $a^n=b^m$ for some nonzero $n,m\in\Z$, then $g(ab)=g(a)g(b)$.
\end{proposition}

\begin{proof}
    We set $x:=g(a), y:=g(b)$ and $z:=g(ab)$ and note that we can freely assume that $n,m\not\in \{1,-1\}$. Moreover, we can assume without loss of generality that $n>0$ and observe that $x^n=y^m$. Suppose by way of contradiction that $z\neq xy$. Then, by Lemma \ref{core}, we have $z=xy^{-1}$ or $z=yx^{-1}$. If $z=xy^{-1}$, then \[y^{n+m}=g(b^{n+m})=g((ab)^n)=z^n=x^ny^{-n}=y^{m-n},\] which implies that both $y$ and $b$ have finite order, a contradiction to our assumption. Similarly, if $z=yx^{-1}$, then \[y^{n+m}=g(b^{n+m})=g((ab)^n)=z^n=y^nx^{-n}=y^{n-m},\] again, a contradiction.
\end{proof}

The next case is that of two elements $a,b\in H$, of which precisely one has infinite order.

\begin{proposition}\label{torsionnontorsion}
    Let $H,K$ be monoids, $f:\Pf(H)\to\Pf(K)$ an isomorphism with pullback $g$ and $a,b\in H$. If $b$ has finite order, then $g(ab)=g(a)g(b)$.
\end{proposition}

\begin{proof}
    We can assume that $b\neq 1_H$ and by \cite[Proposition 4.7]{3Tr-Ya25}, that $a$ has infinite order. Set $x:=g(a)$ and $y:=g(b)$. We then have $\ord(b)=\ord(y)$. Let $M$ denote the monoid generated by $a$ and $b$ and let $G:=M^\times$ be the (finite) cyclic group generated by $b$. Clearly, \[f(G)=f(\{1_H,b\}^{\ord(b)})=\{1_K,y\}^{\ord(y)}=g(G)\] is the cyclic group generated by $y$. Moreover, it is straightforward to verify that for every $X\in\Pf(M)$, there exist positive integers $s,t$ such that \[G\cdot X\cdot \{1_H,a\}^s=G\cdot\{1_H,a\}^{s+t}.\] This implies that $f(X)$ divides $f(G)\cdot\{1_K,x\}^{s+t}$, whence it is contained in the monoid generated by $x$ and $y$. Furthermore, by Lemma \ref{core} (2), we have \[g(a^sb^t)\in \{x^sy^t,x^sy^{-t},x^{-s}y^t\}\] for every $s,t\in\N$. However, by our previous observation, $g(a^sb^t)=x^{-s}y^t$ is impossible, whence \[g(a^sb^t)\in \{x^sy^t,x^{s}y^{-t}\}.\] 

    \noindent \textbf{Claim 1:} \textit{Let $X\in\Pf(G)$. Then $f(X)=g(X)$.} \\

    \noindent \textit{Proof of claim:} Since $X\in\Pf(G)$ and $Y\in\Pf(f(G))$ are equivalent to $X$ dividing $G$ in $\Pf(H)$ and $Y$ dividing $f(G)$ in $\Pf(K)$ respectively, we can easily deduce that $f$ restricts to an isomorphism between $\Pf(G)$ and $\Pf(f(G))$. Since $g^{-1}$ restricts to an isomorphism between $f(G)$ and $G$, it is then straightforward to verify that the map \[F:\Pf(f(G))\to \Pf(G),\] \[F(Y)=g^{-1}(Y)\] is itself an isomorphism with pullback ${g^{-1}}_{|f(G)}$. It follows that \[F\circ f_{|\Pf(G)}\] is an automorphism of $\Pf(G)$ with its pullback being the identity map.  Since $G$ is a finite cyclic group, we infer from \cite[Corollary 10]{2Ra25}, that any automorphism of $\Pf(G)$, with its pullback being the identity automorphism, is itself the identity automorphism. Hence $f_{|\Pf(G)}=F^{-1}$ and $f(X)=g(X)$ for every $X\in\Pf(G)$, proving our claim.  \\

\noindent \textbf{Claim 2:} \textit{Let $X\in\Pf(M)$. Then $f(G)\cap f(X)=g(G\cap X)$.} \\
    
    \noindent \textit{Proof of claim:} Set $Y:=\{1_H\}\cup aG$. Then the equation \[G\cdot Y=G\cdot \{1_H,a\}\] implies that $f(Y)\subseteq f(G)\cdot\{1_H,x\}$ and $f(Y)\not\subseteq f(G)$. Moreover, we have \[\{1_H,ab^s\}\cdot Y=\{1_H,ab^t\}\cdot Y=\{1_H\}\cup aG\cup a^2G\] for every $s,t\in \N$. Then, from our earlier observation that $g(ab^s)\in\{xy^s,xy^{-s}\}$ and the bijectivity of $g$, it follows that \[\{1_K,xy^s\}\cdot f(Y)=\{1_K,xy^t\}\cdot f(Y)\] for every $s,t\in\N$. Consequently, $xf(G)\subseteq f(Y)$ and since by Lemma \ref{card}, $f$ is cardinality-preserving, we conclude that $f(Y)=\{1_K\} \cup xf(G)$. \\

    Let now $X\in\Pf(M)$ and let $k$ be the maximal integer such that $X\cap a^kG\neq \emptyset$. By Claim 1, we can assume that $X\not\subseteq G$, whence $k\geq 1$. Then the equation \[G\cdot \{1_H,a\}^k\cdot X=G\cdot\{1_H,a\}^{2k}\] clearly implies that $k$ is the maximal integer such that $f(X)\cap x^kf(G)\neq \emptyset$. From \[Y^k\cdot X= Y^{2k}\cup (G\cap X)=Y^{2k}\cdot (G\cap X)\] and Claim 1, it then follows that  \[f(Y)^k\cdot f(X)=f(Y)^{2k}\cdot f(G\cap X)=f(Y)^{2k}\cdot g(G\cap X).\] Moreover, keeping in mind that $f(Y)=\{1_K\} \cup xf(G)$, we have \[f(G)\cap f(X)=f(G)\cap (f(Y)^k\cdot f(X))=f(G)\cap (f(Y)^{2k}\cdot g(G\cap X))=g(G\cap X),\] proving the claim. \\
    
    Suppose now towards a contradiction that $g(ab)\neq xy$, whence by Lemma \ref{core} (1), we have $\ord(b)\geq 3$. Then, by our earlier observation, we have $g(ab)=xy^{-1}$, which clearly also implies that $g(ab^{-1})=xy$. For any integer $k$, not divisible by $\ord(b)$, we set $G_k:=G\setminus \{b^k\}$, $f(G)_k:=f(G)\setminus\{y^k\}$ and note that by Claim 1, we have $f(G_k)=g(G_k)=f(G)_k$. Moreover, we set \[X_1:=G_{-1}\cup\{a\},\] \[X_2:=G_1\cup\{a\},\] \[X_3:=G_1\cup\{ab\}\] and \[X_4:=G_{-2}\cup\{ab^{-1}\}.\] Note that \[G\cdot X_1=G\cdot\{1_H,a\},\] whence $f(X_1)\subseteq f(G)\cdot\{1_K,x\}$. Furthermore, by Claim 2 and Lemma \ref{card} we obtain \[f(X_1)=f(G)_{-1}\cup \{xy^s\}\] for some positive integer $s$. 
    However, $ab$ is not a quotient of $X_1$, whence, by Lemma \ref{quot}, $g(ab)=xy^{-1}$ is not a quotient of $f(X_1)$. This leaves us with \[f(X_1)=f(G)_{-1}\cup \{xy^{-2}\}.\] Similarly, by arguing that $ab^{-1}$ is not a quotient of $X_2$, $a$ is not a quotient of $X_3$ and that $ab$ is not a quotient of $X_4$, we obtain \[f(X_2)=f(G)_1\cup\{xy^2\},\] \[f(X_3)=f(G)_1\cup\{xy\}\] and\[f(X_4)=f(G)_{-2}\cup \{xy^{-3}\}.\] Then, keeping in mind that $\ord(b)=\ord(y)\geq 3$, the equations \[X_1\cdot X_2=G\cup aG\cup \{a^2\}=X_3\cdot X_4,\] \[f(X_1)\cdot f(X_2)=f(G)\cup xf(G)\cup\{x^2\}\] and \[f(X_3)\cdot f(X_4)=f(G)\cup xf(G)\cup\{x^2y^{-2}\}\] imply that $x^2=x^2y^{-2}$, a contradiction to $\ord(y)\geq 3$.

\end{proof}

\begin{corollary}\label{independent}
    Let $H,K$ be monoids, $f:\Pf(H)\to\Pf(K)$ an isomorphism with pullback $g$ and $a,b\in H$. Then $g(ab)\neq g(a)g(b)$ only if $a$ and $b$ are independent elements.
\end{corollary}

\begin{proof}
    This follows from Propositions \ref{relation} and \ref{torsionnontorsion}.  
\end{proof}

We now turn our attention to elements $a,b\in H$ that are independent.

\begin{lemma} \label{powersindependent}
    Let $H,K$ be monoids, $f:\Pf(H)\to\Pf(K)$ an isomorphism with pullback $g$ and $a,b\in H$. If $g(ab)\neq g(a)g(b)$, then $g(a^nb^m)\neq g(a)^ng(b)^m$ for every $n,m\in\N$.
\end{lemma}

\begin{proof}
    Set $x:=g(a)$ and $y:=g(b)$ and suppose that $g(ab)\neq xy$. Then, by Lemma \ref{core}, we have $g(ab)\in \{xy^{-1},yx^{-1}\}$ and we can assume without loss of generality that $g(ab)=xy^{-1}$. We will now show that \[g(a^nb^m)\in \{x^ny^{-m},y^mx^{-n}\}\] for every $n,m\in \N$ by induction on $n$. If $n=1$, we can freely assume that $m\geq 2$. By Lemma \ref{core} and our assumption that $g(ab)=xy^{-1}$, we then have \[g(ab^m)\in\{xy^m,xy^{-m},y^mx^{-1}\}\] and \[g((ab)b^{m-1})\in\{xy^{m-2},xy^{-m},y^{m}x^{-1}\}.\] Since by Corollary \ref{independent}, $a$ and $b$ are independent, so are $x$ and $y$, which leaves us with \[g(ab^m)\in \{xy^{-m},y^{m}x^{-1}\}.\] Suppose now that $n\geq 2$. Then, by Lemma \ref{core} and the induction hypothesis, we obtain \[g(a^nb^m)\in\{x^ny^m,x^ny^{-m},y^mx^{-n}\}\] and \[g(a (a^{n-1}b^m))\in \{x^ny^{-m},x^{2-n}y^m,x^{n-2}y^{-m},y^mx^{-n}\}\] and thus \[g(a^nb^m)\in\{x^ny^{-m},y^mx^{-n}\},\] as desired.
\end{proof}

In order to fully understand how the pullback of an isomorphism $f:\Pf(H)\to\Pf(K)$ can fail to be an isomorphism itself, we need another definition. In \cite{Tr-Ya25}, Tringali and Yan proved that the only non-trivial automorphism of $\P_{\fin,0}(\N_0)$ is the \textit{reversion map} \[\rev:\P_{\fin,0}(\N_0)\to\P_{\fin,0}(\N_0),\] defined by \[\rev(X)=\max X-X.\] Let now $H,K$ be monoids, $f:\Pf(H)\to\Pf(K)$ an isomorphism with pullback $g$, $a\in H$ an element of infinite order and $x=g(a)\in K$. We denote by $H_a$ and $K_x$ the submonoids of $H$ and $K$ generated by $a$ and $x$ respectively. It is then easy to verify that for every $X\in\Pf(H_a)$, there is $n\in\N$ such that $X$ divides $\{1_H,a\}^n$, whence $f(X)\in \Pf(K_x)$. By a symmetric argument, involving $f^{-1}$, we infer that $f$ restricts to an isomorphism $\Pf(H_a)\to\Pf(K_x)$. 
Then, since both $H_a$ and $K_x$ are isomorphic to $(\N_0,+)$ in a unique way, we conclude that the restriction of $f$ to $\Pf(H_a)$ either corresponds to the identity automorphism on $\P_{\fin,0}(\N_0)$ or to the reversion map. The former is equivalent to \[f(\{1_H,a,a^3\})=\{1_K,x,x^3\}\] and the latter, to \[f(\{1_H,a,a^3\})=\{1_K,x^2,x^3\},\] in which case we say that $a$ is \textit{reversed under} $f$. It is clear that an element $a\in H$ is reversed under $f$ if and only if $g(a)$ is reversed under $f^{-1}$. Moreover, $a$ is reversed under $f$ if and only if $a^k$ is reversed under $f$ for every $k\in\N$. We denote by  $H_{R,f}$ the set containing $1_H$ together with all elements of $H$ that are reversed under $f$, we set $H_{N,f}:=H\setminus H_{R,f}$ and whenever the isomorphism $f$ is clear from the context, we simply write $H_R$ and $H_N$. We continue by showing that only non-units can be reversed under $f$.

\begin{lemma} \label{unitsnotrev}
    Let $H,K$ be monoids and $f:\Pf(H)\to\Pf(K)$ an isomorphism with pullback $g$. If $a\in H^\times$ is a non-trivial unit of $H$ of infinite order, then $a$ is not reversed under $f$.
\end{lemma}

\begin{proof}
  Set $x:=g(a)$, $X:=\{1_H,a,a^{-2}\}$ and suppose by way of contradiction that $a$ is reversed under $f$. Then the equation \[X\cdot\{1_H,a,a^3\}=\{1_H,a\}^4\cdot\{1_H,a^{-1}\}^2\] implies that \[f(X)\cdot\{1_K,x^2,x^3\}=\{1_K,x\}^4\cdot\{1_K,x^{-1}\}^2,\] from which it clearly follows that \[\{1_K,x,x^{-1},x^{-2}\}\subseteq f(X).\] However, by Lemma \ref{card}, we have $|f(X)|=3$, a contradiction.
\end{proof}

The following proposition shows how the pullback may fail to be an isomorphism. 

\begin{proposition} \label{onereversed}
    Let $H,K$ be monoids, $f:\Pf(H)\to\Pf(K)$ an isomorphism with pullback $g$ and $a,b\in H$. Then $g(ab)\neq g(a)g(b)$ if and only if precisely one of $a$ and $b$ is reversed under $f$.
\end{proposition}

\begin{proof}
    Set $x:=g(a), y:=g(b)$ and suppose first that $g(ab)\neq xy$. Then, by Corollary \ref{independent}, $a$ and $b$ are independent, whence $x$ and $y$ are independent as well and by Lemma \ref{core}, we can assume without loss of generality that $g(ab)=xy^{-1}$. We set $X:=\{1_H,a,ab\}$, $Y:=\{1_H,b,ab\}$ and observe that the quotients of $X$ and $Y$ are contained in the set \[\{a,a^{-1},b,b^{-1},ab,a^{-1}b^{-1}\}.\] Hence, by Lemma \ref{quot}, we obtain \[f(X),f(Y)\subseteq \{1_K,x,x^{-1},y,y^{-1},xy^{-1},yx^{-1}\}\] and from Lemma \ref{card}, it follows that $|f(X)|=|f(Y)|=3$. Moreover, since the equation \[X\cdot Y=\{1_H,a\}\cdot\{1_H,ab\}\cdot\{1_H,b\}\] implies that both $f(X)$ and $f(Y)$ divide \[\{1_K,x\}\cdot\{1_K,xy^{-1}\}\cdot\{1_K,y\},\] we conclude that \[f(X),f(Y)\subseteq \{1_K,x,y,xy^{-1}\}.\] 
    It is now straightforward to verify that \[f(X)\cdot f(Y)=\{1_K,x\}\cdot\{1_K,xy^{-1}\}\cdot\{1_K,y\}\] is only possible if $f(X)=\{1_K,x,y\}$ and $f(Y)=\{1_K,x,xy^{-1}\}$ or if $f(X)=\{1_K,x,xy^{-1}\}$ and $f(Y)=\{1_K,x,y\}$. Suppose first that $f(Y)=\{1_K,x,y\}$ and by way of contradiction, that $b$ is not reversed under $f$. Note that $ab^3$ is a quotient of \[\{1_H,b,b^3\}\cdot Y\] of multiplicity $1$, whereas by Lemmas \ref{core} and \ref{powersindependent}, we have $g(ab^3)\in\{xy^{-3},y^3x^{-1}\}$, whence $g(ab^3)$ is a quotient of \[\{1_K,y,y^3\}\cdot f(Y)\] of multiplicity $2$, contradicting Lemma \ref{quot}. Hence $b$ is reversed under $f$.
    In a similar fashion, we suppose by contradiction that $a$ is reversed under $f$. Then $a^3b$ is a quotient of \[\{1_H,a,a^3\}\cdot Y\] of multiplicity $2$, whereas $g(a^3b)$ is a quotient of \[\{1_K,x^2,x^3\}\cdot f(Y)\] of multiplicity $1$, another contradiction. If, on the other hand, we have $f(X)=\{1_K,x,y\}$, then a symmetric argument yields that $a$ is reversed under $f$ and $b$ is not. Hence we have shown that if $g(ab)\neq xy$, then precisely one of $a$ and $b$ is reversed under $f$. \\

    Conversely, suppose without loss of generality that $b$ is reversed under $f$ and $a$ is not. Moreover, we suppose towards a contradiction that $g(ab)=xy$. Using a similar approach as before, we argue that the quotients of $X$ and $Y$ are contained in \[\{a,a^{-1},b,b^{-1},ab,a^{-1}b^{-1}\},\] whence  \[f(X),f(Y)\subseteq \{1_K,x,x^{-1},y,y^{-1},xy,x^{-1}y^{-1}\}.\] Then the equations \[X\cdot Y=\{1_H,a\}\cdot\{1_H,ab\}\cdot\{1_H,b\}\] and \[f(X)\cdot f(Y)=\{1_K,x\}\cdot\{1_K,xy\}\cdot\{1_K,y\}\] imply that $f(X)=\{1_K,x,xy\}$ and $f(Y)=\{1_K,y,xy\}$ or that $f(X)=\{1_K,y,xy\}$ and $f(Y)=\{1_K,x,xy\}$. If $f(X)=\{1_K,x,xy\}$, then $ab^3$ is a quotient of \[\{1_H,b,b^3\}\cdot X\] of multiplicity $2$, whereas by Lemma \ref{core}, $g(ab^3)\in\{xy^3,xy^{-3},y^3x^{-1}\}$ is a quotient of \[\{1_K,y^2,y^3\}\cdot f(X)\] of multiplicity $1$, a contradiction. 
    Hence $f(X)=\{1_K,y,xy\}$ and $a^3b$ is a quotient of \[\{1_H,a,a^3\}\cdot X\] of multiplicity $1$, while $g(a^3b)\in\{x^3y,x^3y^{-1},yx^{-3}\}$ is a quotient of \[\{1_K,x,x^3\}\cdot f(X)\] of multiplicity $2$, another contradiction.
\end{proof}

We will now prove a much stronger version of Lemma \ref{unitsnotrev}.

\begin{lemma} \label{nothingreversed}
    Let $H,K$ be monoids and $f:\Pf(H)\to\Pf(K)$ an isomorphism with pullback $g$. If $H^\times\neq\{1_H\}$, then no element of $H$ is reversed under $f$.
\end{lemma}

\begin{proof}
   Set $x:=g(a), y:=g(b)$, let $b\in H^\times$ be a non-trivial unit of $H$ and suppose by way of contradiction that there is an element $a\in H$ of infinite order that is reversed under $f$. We first consider the case, where $b$ has finite order. Let $G$ denote the cyclic group generated by $b$, $M$ the monoid generated by $a$ and $b$ and note that $f(G)=g(G)$ is the cyclic group generated by $y$.  We recall from the proof of Proposition \ref{torsionnontorsion} that for any $X\in\Pf(M)$, we have \[f(G)\cap f(X)=g(G\cap X).\] Moreover, we set \[X:=\{1_H\}\cup aG\cup a^3G\] and \[Y:=\{1_H\}\cup aG\cup a^2G\cup a^3G.\] Then the equations \[G\cdot X=G\cdot\{1_H,a,a^3\}\] and \[G\cdot Y=G\cdot\{1_H,a\}^3\] imply that \[f(G)\cdot f(X)=f(G)\cdot \{1_K,x^2,x^3\}\] and \[f(G)\cdot f(Y)=f(G)\cdot\{1_K,x\}^3.\] Hence, by Lemma \ref{card} and the fact that $\ord(b)=\ord(y)$, it follows that \[f(X)=\{1_K\}\cup x^2G\cup x^3G\] and \[f(Y)=\{1_K\}\cup xf(G)\cup x^2f(G)\cup x^3f(G).\] 
   However, the equation \[\{1_H,a\}\cdot X=\{1_H,a\}\cdot Y\] then implies that \[\{1_K,x\}\cdot f(X)=\{1_K,x\}\cdot f(Y),\] which, due to the fact that \[xy\in \{1_K,x\}\cdot f(Y)\] and \[xy\not\in \{1_K,x\}\cdot f(X),\] leads to a contradiction. \\

   Suppose on the other hand that $b$ has infinite order and set $X:=\{1_H,a,ab\}$. By Lemma \ref{unitsnotrev}, $b$ is not reversed under $f$, whence $g(ab)\neq xy$ by Proposition \ref{onereversed}. Similarly, since $a$ is reversed under $f$, it cannot be a unit of $H$, which implies that $x$ cannot be a unit of $K$. Hence by Lemma \ref{core}, we obtain $g(ab)=xy^{-1}$. In particular, $a$ and $b$ are independent elements by Corollary \ref{independent}, whence $x$ and $y$ are independent as well. By repeating the arguments in the first part of the proof of Proposition \ref{onereversed} and considering that $a$ is reversed under $f$, we infer that $f(X)=\{1_K,x,y\}$. Then from the equation \[X\cdot \{1_H,a,b^{-1}\}=\{1_H,a\}\cdot\{1_H,b^{-1}\}\cdot\{1_H,ab\},\] it follows that \[y\in f(X)\subseteq \{1_K,x\}\cdot\{1_K,y^{-1}\}\cdot\{1_K,xy^{-1}\},\] contradicting the fact that $x$ and $y$ are independent elements.
\end{proof}

As a consequence, we obtain a solution to the isomorphism problem for a large class of monoids.

\begin{theorem}\label{nonreducedpositive}
    Let $H,K$ be monoids and $f:\Pf(H)\to\Pf(K)$ an isomorphism with pullback $g$. If $H^\times\neq\{1_H\}$ or $K^\times\neq \{1_K\}$, then $g$ is an isomorphism. In particular, the isomorphism problem has a positive answer for the class of commutative, cancellative monoids with non-trivial unit group.
\end{theorem}

\begin{proof}
   Clearly, if one of $H$ or $K$ contains a non-trivial unit, so does the other. The assertion then follows from Proposition \ref{onereversed} and Lemma \ref{nothingreversed}.
\end{proof}

\section{The case of reduced monoids}

In this section, we prove our main theorem and fill in the gaps between Theorem \ref{nonreducedpositive} and \cite[Theorem 2]{Ra25}, which deals with the class of reduced valuation monoids. We start by investigating the sets $H_N$ and $H_R$, where $H$ and $K$ are reduced monoids and $f:\Pf(H)\to\Pf(K)$ is an isomorphism. It turns out that these sets are closely related to two subsemigroups of $H$, which provide a canonical decomposition of $H$ into a valuation submonoid and a subsemigroup with certain strong properties. 
Ultimately, this decomposition will allow us to describe the structure of $K$ completely in terms of the structure of $H$. Let $H$ be a reduced monoid and $a\in H$. We call $a$ a \textit{pseudo-unit} of $H$ if for all $b\in H$, either $ab^{-1}\in H$ or $ba^{-1}\in H$. We let $H_v$ denote the set of pseudo-units of $H$ and we set ${H_v}:=H\setminus H_v$. We clearly have $1_H\in H_v$ and \[H={H_v}^c\sqcup H_v.\] In particular, $H$ is a valuation monoid if and only if $H=H_v$ if and only if ${H_v}^c=\emptyset$.

\begin{lemma}\label{pseudo}
    Let $H$ be a reduced monoid. Then ${H_v}^c$ is a subsemigroup of $H$ and $H_v$ is a valuation submonoid of $H$. Moreover, we have ${H_v}^c\cdot \mathsf{q}(H_v)={H_v}^c$.
\end{lemma}

\begin{proof}
If ${H_v}^c=\emptyset$, then all the assertions are vacuously true, whence we assume that ${H_v}^c\neq\emptyset$. We start by showing that $H_v$ is a valuation submonoid of $H$. Let $a,b\in H_v$ and $c\in H$. We aim to show that $abc^{-1}\in H$ or $c(ab)^{-1}\in H$. If $ac^{-1}\in H$, then clearly $abc^{-1}\in H$. Otherwise, since $a\in H_v$, we have $ca^{-1}\in H$. Then, from $b\in H_v$, it follows that $abc^{-1}\in H$ or $c(ab)^{-1}\in H$, which proves that $H_v$ is a submonoid of $H$. Clearly, we also either have $ab^{-1}\in H$ or $ba^{-1}\in H$ and we assume without loss of generality that $ab^{-1}\in H$ with the goal of showing that $a(bc)^{-1}\in H$ or $bca^{-1}\in H$. However, this follows from the fact that $a\in H_v$ and $bc\in H$. Hence $ab^{-1}\in H_v$ and $H_v$ is a valuation monoid. \\

Let now $a,b\in {H_v}^c$ and suppose by way of contradiction that $ab\in H_v$. By definition, there is $c\in H$ such that neither $ac^{-1}\in H$ nor $ca^{-1}\in H$. If $c(ab)^{-1}\in H$, then clearly $ca^{-1}\in H$, a contradiction. Hence, since $ab\in H_v$, we have $abc^{-1}\in H$ and $a^2bc^{-1}\in H$. But then, either $ca^{-1}=(ab)(a^2bc^{-1})^{-1}\in H$ or $ac^{-1}=(a^2bc^{-1})(ab)^{-1}\in H$, another contradiction. Thus, $ab\in {H_v}^c$ and ${H_v}^c$ is a subsemigroup of $H$.\\

We continue by showing that ${H_v}^c\cdot \mathsf{q}(H_v)={H_v}^c$. Suppose that $a\in {H_v}^c$, $b\in H_v$ and by way of contradiction, that $ab\in H_v$. Since $H_v$ is a valuation monoid, we then infer that either $a=(ab)b^{-1}\in H_v$, contradicting our assumption, or $a^{-1}\in H_v$, which is a contradiction to $H$ being reduced. Thus, ${H_v}^c\cdot H_v={H_v}^c$. Similarly, suppose that $ba^{-1}\in H$. Since we already showed that ${H_v}^c\cdot H={H_v}^c$, we obtain $b=a(ba^{-1})\in {H_v}^c$, contradciting our assumption. Hence $ab^{-1}\in H$ and since $H_v$ is a monoid, we conclude that $ab^{-1}\in{H_v}^c$. This proves that ${H_v}^c\cdot (H_v)^{-1}={H_v}^c$, from which we infer that ${H_v}^c\cdot \mathsf{q}(H_v)={H_v}^c$.
\end{proof}

The following lemma shows that the sets $H_N$ and $H_R$ exhibit quite similar properties to those of $H_v$ and ${H_v}^c$.

\begin{lemma} \label{split monoids}
    Let $H,K$ be reduced monoids and $f:\Pf(H)\to\Pf(K)$ an isomorphism with pullback $g$. Then $H_N$ is a subsemigroup of $H$, $H_R$ is a submonoid of $H$ and $H_R\subseteq H_v$.
\end{lemma}

\begin{proof}
    Let $a,b\in H$ and set $x:=g(a)$ and $y:=g(b)$. We first assume that $H_N\neq\emptyset$, $a,b\in H_N$ and aim to show that $ab\in H_N$. Suppose towards a contradiction that $ab\in H_R$. Since $a\neq 1_H\neq b$ and $H$ is reduced, we have $ab\neq 1_H$, whence $ab$ is reversed under $f$. Note that $b^2\in H_N$, which, together with Proposition \ref{onereversed} implies that $g(ab^2)=xy^2$. Moreover, since $a,b\in H_N$, we have $g(ab)=xy$. However, since $ab$ is reversed under $f$, we also obtain \[g((ab)\cdot b)\in\{x,x^{-1}\},\] a contradiction. Hence $H_N$ is a subsemigroup of $H$ and by a symmetrical argument, $H_R\setminus\{1_H\}$ is also a subsemigroup of $H$, whence $H_R$ is a submonoid of $H$.\\

    It remains to show that $H_R\subseteq H_v$. This is obvious if $H_R=\{1_H\}$, whence we suppose that $H_R$ is non-trivial. Let $a\in H_R\setminus\{1_H\}$, $b\in H\setminus \{1_H\}$ and set $X:=\{1_K,a,b\}$. Then, since $H$ is reduced, the set of quotients of $X$ is contained in \[\{a,b,ab^{-1},ba^{-1}\}.\] Suppose towards a contradiction that neither $ab^{-1}\in H$, nor $ba^{-1}\in H$. Then \[f(X)=\{1_K,x,y\}\] by Lemmas \ref{card} and \ref{quot} and $a^2b$ is a quotient of \[\{1_H,a,a^3\}\cdot X\] of multiplicity $1$. However, since $g(a^3b)\in\{x^3y,x^3y^{-1},yx^{-3}\}$ by Lemma \ref{core}, the element $g(a^3b)$ is a quotient of \[\{1_K,x^2,x^3\}\cdot f(X)\] of multiplicity $2$, contradicting Lemma \ref{quot}. Hence either $ab^{-1}\in H$ or $ba^{-1}\in H$ and $H_R\subseteq H_v$, proving our assertion.

\end{proof}

For two monoids $H,K$ and an isomorphism $f:\Pf(H)\to \Pf(K)$, we are now able to describe $K$ as a submonoid of $\mathsf{q}(H)$ in terms of the sets $H_N$ and $H_R$.

\begin{proposition}\label{decomposition}
    Let $H,K$ be reduced monoids and $f:\Pf(H)\to\Pf(K)$ an isomorphism with pullback $g$. Then $H_N\sqcup {H_R}^{-1}$ is a submonoid of $\mathsf{q}(H)$ and $K\simeq H_N\sqcup {H_R}^{-1}$.
\end{proposition}

\begin{proof}
  Let $a,b\in H$ and set $x:=g(a)$ and $y:=g(b)$. We start by showing that $H_N\sqcup {H_R}^{-1}$ is a submonoid of $\mathsf{q}(H)$. By Lemma \ref{split monoids}, both $H_N$ and ${H_R}^{-1}$ are subsemigroups of $\mathsf{q}(H)$. Hence we can assume that $H_N\neq \emptyset$ and $H_R\neq\{1_H\}$ with the aim to show that $ab^{-1}\in H_N\sqcup {H_R}^{-1}$, where $a\in H_N$ and $b\in H_R\setminus\{1_H\}$. By Lemma \ref{split monoids}, we have $H_R\subseteq H_v$, whence either $ab^{-1}\in H$ or $ba^{-1}\in H$. Suppose first that $ab^{-1}\in H$. Then, since $H_R$ is a submonoid of $H$ by Lemma \ref{split monoids}, it follows that $ab^{-1}\in H_N\subseteq H_N \sqcup {H_R}^{-1}$.
   If, on the other hand, $ba^{-1}\in H$, then $ba^{-1}\in H_R$ due to the fact that $H_N$ is a subsemigroup of $H$. Hence $ab^{-1}\in H_N\sqcup {H_R}^{-1}$, as desired. \\

  We now define the map $h: H_N \sqcup {H_R}^{-1}\to K$ in the following way. If $a\in H_N$, then $h(a)=g(a)$ and $h(a)=g(a^{-1})$ otherwise. It is straightforward to verify that $h(1_H)=1_K$ and that $h$ is bijective. By Lemma \ref{split monoids}, both $H_N$ and $H_R$ are subsemigroups of $H$, whence by Proposition \ref{onereversed}, we have \[h(ab)=g(ab)=g(a)g(b)=h(a)h(b)\] whenever $a,b\in H_N$. Similarly, if $a,b\in H_R\setminus\{1_H\}$, then \[h(a^{-1}b^{-1})=g(ab)=g(a)g(b)=h(a^{-1})h(b^{-1}).\] If, on the other hand, $a\in H_N$ and $b\in H_R\setminus\{1_H\}$, we recall from the first part of the proof that either $ab^{-1}\in H$, in which case $ab^{-1}\in H_N$, or $ba^{-1}\in H$, in which case $ba^{-1}\in H_R$. Moreover, since $y$ is reversed under $f^{-1}$ and $x$ is not, we obtain \[g^{-1}(xy)\in\{ab^{-1},ba^{-1}\}\] by Proposition \ref{onereversed}. Then, if $ab^{-1}\in H\cap H_N$ and $ba^{-1}\not\in H$, we have\[h(ab^{-1})=g(ab^{-1})=xy=g(a)g(b)=h(a)h(b^{-1}).\] 
  Otherwise, we have $ba^{-1}\in H\cap H_R$ and $ab^{-1}\not\in H$, from which it follows that \[h(ab^{-1})=g(ba^{-1})=xy=g(a)g(b)=h(a)h(b^{-1}).\] Hence $h$ is an isomorphism and $K\simeq H_N\sqcup {H_R}^{-1}$, as desired.

\end{proof}

We are now in a position to prove our main theorem.

\begin{theorem}\label{main}
    Let $H,K$ be commutative and cancellative monoids. Then $\Pf(H)\simeq\Pf(K)$ if and only if one of the two following statements hold. \begin{enumerate}
        \item $H$ is isomorphic to $K$.
        \item $H$ and $K$ are reduced and $K$ is isomorphic to a submonoid of $\mathsf{q}(H)$ of the form ${H_v}^c\sqcup \tilde{H}$, where $\tilde{H}$ is a reduced valuation monoid with $\mathsf{q}(H_v)=\mathsf{q}(\tilde{H})$. Here, $H_v$ denotes the monoid of pseudo-units of $H$ and ${H_v}^c$ its complement in $H$.
    \end{enumerate}
\end{theorem}

\begin{proof}
    Let $f:\Pf(H)\to\Pf(K)$ be an isomorphism with pullback $g$. By Theorem \ref{nonreducedpositive}, $H$ and $K$ are isomorphic whenever $H$ or $K$ contains a non-trivial unit. If, on the other hand, both $H$ and $K$ are reduced, we infer by Proposition \ref{decomposition}, that $K\simeq H_N\sqcup {H_R}^{-1}$. Clearly, $H_N\sqcup {H_R}^{-1}$ is a submonoid of $\mathsf{q}(H)$. Since $H_R\subseteq H_v$ by Lemma \ref{split monoids} and $H=H_N\sqcup H_R={H_v}^c\sqcup H_v$ by definition, it also follows that ${H_v}^c\subseteq H_N$. We then write \[H_N\sqcup {H_R}^{-1}={H_v}^c\sqcup \tilde{H},\] where $\tilde{H}:=((H_N\cap H_v)\sqcup {H_R}^{-1})$ and we aim to show that $\tilde{H}$ is a reduced valuation monoid with quotient group equal to $\mathsf{q}(H_v)$. Since, by Lemma \ref{pseudo}, $H_v=((H_N\cap H_v)\sqcup {H_R})$ is a reduced valuation monoid, this follows easily, once we prove that $\tilde{H}$ is a monoid. 
    Note that both $H_N\cap H_v$ and $H_R^{-1}$ are semigroups by Lemmas \ref{pseudo} and \ref{split monoids}, whence it remains to show that $ab^{-1}\in \tilde{H}$, whenever $H_N\cap H_v\neq\emptyset$, $H_R\neq\{1_H\}$, $a\in H_N\cap H_v$ and $b\in H_R\setminus\{1_H\}$. Suppose the contrary. Then, by Proposition \ref{decomposition}, $H_N\sqcup {H_R}^{-1}={H_v}^c\sqcup \tilde{H}$ is a monoid, which implies that $ab^{-1}\in {H_v}^c\subseteq H$. However, since $H_v$ is a valuation monoid and $a,b\in H_v$, we either have $ab^{-1}\in H_v$ or $ba^{-1}\in H_v$, both being impossible since $H$ is reduced. Hence $\tilde{H}$ is a monoid, as desired. \\

    If conversely, statement (1) holds and $H\simeq K$, then it follows trivially that $\Pf(H)\simeq\Pf(K)$.\\
    
    Suppose therefore that statement (2) holds. Since $\Pf(H)\simeq \Pf(K)$ is equivalent to $\Pf(H)\simeq \Pf({H_v}^c\sqcup \tilde{H})$, we may replace $K$ with an isomorphic copy of itself and assume without loss of generality that $K={H_v}^c\sqcup \tilde{H}$. If ${H_v}^c=\emptyset$, then both $H$ and $K$ are reduced valuation monoids with equal quotient groups and it follows from \cite[Theorem 2]{Ra25} that $\Pf(H)\simeq \Pf(K)$. If, on the other hand $\tilde{H}=\{1_K\}$, then $H_v=\{1_H\}$ and $H=K$. In particular, $\Pf(H)\simeq\Pf(K)$ and we can therefore assume that ${H_v}^c$ is nonempty and both $H_v$ and $\tilde{H}$ are non-trivial. Let $a\in {H_v}^c$ and suppose towards a contradiction that $a\in K_v$. 
    By definition of ${H_v}^c$, we find $b\in{H_v}^c$ such that neither $ab^{-1}\in H$ nor $ba^{-1}\in H$. Since $a\in K_v$ and $b\in K$, we have either $ab^{-1}\in K$ or $ba^{-1}\in K$. If $ab^{-1}\in K$, then clearly $ab^{-1}\not\in{H_v}^c\subseteq H$, whence $ab^{-1}\in \tilde{H}$. However, since $H_v$ is a valuation monoid with $\mathsf{q}(H_v)=\mathsf{q}(\tilde{H})$, we obtain $ab^{-1}\in H_v$ or $ba^{-1}\in H_v$, a contradiction. By a symmetric argument, $ba^{-1}\in K$ then also leads to a contradiction. Hence ${H_v}^c\subseteq {K_v}^c$. \\

    Let now $a\in \tilde{H}$. Since $\tilde{H}$ is a valuation monoid, we have $ab^{-1}\in K$ or $ba^{-1}\in K$ for all $b\in \tilde{H}$. Moreover, if $b\in {H_v}^c$, then due to $a^{-1}\in \mathsf{q}(H_v)$ and Lemma \ref{pseudo}, we obtain $ba^{-1}\in {H_v}^c\subseteq K$. Hence $\tilde{H}\subseteq K_v$ and since ${H_v}^c\subseteq {K_v}^c,$ we conclude that $K_v=\tilde{H}$ and ${K_v}^c={H_v}^c$.
    We now prove a claim that is a slight variation of \cite[Lemma 1]{Ra25}. \\

    \noindent \textbf{Claim:} Let $X\in \Pf(H)$ and $Y\in\Pf(K)$. Then there are unique elements $a,b\in \mathsf{q}(H)=\mathsf{q}(K)$ such that  $aX\in\Pf(K)$ and $bY\in\Pf(H)$.\\

    \noindent \textit{Proof of Claim:} We start by showing that there is at most one element $a\in\mathsf{q}(H)$ such that $aX\in\Pf(K)$. Suppose the contrary and let $a,c\in\mathsf{q}(H)$ be distinct elements such that $aX,cX\in\Pf(K)$. Since $1_K\in aX$ and $1_K\in cX$, we obtain $a^{-1},c^{-1}\in X$. Then $ac^{-1}\in aX\subseteq K$ and $ca^{-1}\in cX\subseteq K$, contradicting the fact that $K$ is reduced. To prove the existence, we consider $X\cap H_v\in\Pf(H_v)$. Since both $H_v$ and $K_v$ are reduced valuation monoids with equal quotient groups, by \cite[Lemma 1]{Ra25}, we find a unique $a\in \mathsf{q}(H_v)$ such that $a(X\cap H_v)\in\Pf(K_v)$. It then follows from Lemma \ref{pseudo}, that $a(X\cap {H_v}^c)\subseteq {H_v}^c$ and consequently, that $aX\subseteq K$. The other statement is then proved analogously by swapping the roles of $H_v$ and $K_v$, which finishes the proof of the claim.\\

   The rest of the proof follows the same steps as the proof of \cite[Theorem 2]{Ra25}. We define a map $f:\Pf(H)\to\Pf(K)$ via $f(X)=aX$, where $a\in\mathsf{q}(H)$ such that $aX\in\Pf(K)$. By the previous claim, $f$ is both well-defined and bijective, with $f^{-1}$ defined analogously. Moreover, if $X,Y\in\Pf(H)$, and $f(X)=aX, f(Y)=bY,f(XY)=cXY$ for some $a,b,c\in\mathsf{q}(H)$, then $f(X)f(Y)=abXY\in\Pf(K)$. By the uniqueness of the element $c$, we obtain $ab=c$ and in conclusion, $f(XY)=f(X)f(Y)$, proving that $f$ is an isomorphism between $\Pf(H)$ and $\Pf(K)$.

\end{proof}

To conclude this paper, we present an example of two non-isomorphic monoids, which are not valuation monoids, with isomorphic reduced finitary power monoids. We use a similar approach as in \cite[Corollary 3]{Ra25}, where such an example is given for reduced valuation monoids. Note that the monoids $(\Z\times\N)\cup(\N_0\times\{0\})$ and $\{(x,y)\in\Z^2:y\leq\alpha x\},$ for some $\alpha\in \mathbb{R}\setminus\mathbb{Q}$ are both reduced valuation submonoids of $(\Z^2,+)$. Moreover, the element $(1,0)$ is irreducible in $(\Z\times\N)\cup(\N_0\times\{0\})$, while by item B4 of \cite[Theorem 10]{Le22}, $\{(x,y)\in\Z^2:y\leq\alpha x\}$ does not contain any irreducible elements. 

\begin{example}
    Let $G$ be the free abelian group of rank 4 with generators $a,b,c,d$ and let $\tilde{G}$ be the subgroup of $G$, generated by $a$ and $b$. Since $\tilde{G}\simeq (\Z^2,+)$, we can define $\tilde{H}$ to be a reduced valuation submonoid of $\tilde{G}$, isomorphic to \[(\Z\times\N)\cup(\N_0\times\{0\})\] and, similarly, $\tilde{K}$ to be a reduced valuation submonoid of $\tilde{G}$, isomorphic to \[\{(x,y)\in\Z^2:y\leq\alpha x\},\] where $\alpha\in \mathbb{R}\setminus\mathbb{Q}$. We now let $M$ denote the subsemigroup of $G$, generated by $c$ and $d$, and define the reduced monoids \[H:=\tilde{H}\sqcup  (\tilde{G}\cdot M) \] and \[K:=\tilde{K} \sqcup (\tilde{G}\cdot M).\] We claim that $H_v=\tilde{H}$. Indeed, by construction, we have $\tilde{H}\subseteq H_v$. Moreover, since every element in $H$ can be expressed uniquely as a reduced word in the symbols $a,b,c,d$ and no such word contains a negative power of $c$ or $d$, we conclude that $\tilde{G}\cdot M\subseteq {H_v}^c$.
    In a similar fashion, we obtain $K_v=\tilde{K}$, which implies that \[K={H_v}^c\sqcup \tilde{K}.\] Then, due to the fact that $\tilde{H}$ and $\tilde{K}$ are reduced valuation monoids with equal quotient groups, we can apply Theorem \ref{main} to conclude that $\Pf(H)\simeq\Pf(K)$. However, $\tilde{H}$ contains an irreducible element, which is also irreducible in $H$, whereas $K$ does not contain any irreducible elements, from which we infer that $H$ and $K$ are not isomorphic. In conclusion, we have constructed two non-isomorphic non-valuation monoids, with isomorphic reduced finitary power monoids.
\end{example}


\begin{thebibliography}{10}

    \bibitem{An-Tr21}
    A. Antoniou and S. Tringali, \emph{On the arithmetic of power monoids and sumsets in cyclic groups}, Pacific J. Math. \textbf{312} (2021), 279--308      
    
    \bibitem{Bi-Ge25}
    P.-Y. Bienvenu and A. Geroldinger, \emph{On algebraic properties of power monoids of numerical monoids}, Isr. J. Math. \textbf{265} (2025), 867--900.
    
    \bibitem{Co-Tr25}
    L. Cossu and S. Tringali, \emph{On the arithmetic of power monoids}, J. Algebra \textbf{686} (2026), 793--813.
    
    \bibitem{Fa-Tr18}
    Y. Fan and S. Tringali, \emph{Power monoids: A bridge between Factorization Theory and Arithmetic Combinatorics}, J. Algebra \textbf{512} (2018), 252--294.
    
    \bibitem{Ga-Zh14}
    A. Gan and X. Zhao, \emph{Global determinism of Clifford semigroups}, J. Aust. Math. Soc. \textbf{97} (2014), 63--77.
    
    \bibitem{Go-Is84}
    M. Gould and J. A. Iskra, \emph{Globally determined classes of semigroups}, Semigroup Forum \textbf{28} (1984), 1-11.

    \bibitem{Le22}
    G. Lettl, \emph{Atoms of root-closed submonoids of $\Z^2$}, Semigroup Forum \textbf{105} (2022), 282--294.

    
    \bibitem{Mo73}
    E.M. Mogiljanskaja, \emph{Non-isomorphic semigroups with isomorphic semigroups of subsets}, Semigroup Forum \textbf{6} (1973), 330--333.

    \bibitem{Os68}
    H.-H. Ostmann, \emph{Additive Zahlentheorie, 1. Teil: Allgemeine Untersuchungen}, Springer-Verlag, 1968. 

    \bibitem{Ra25}
    B. Rago, \emph{A counterexample to an isomorphism problem for power monoids}, Proc. Amer. Math. Soc., to appear.

    \bibitem{2Ra25}
    B. Rago, \emph{The automorphism group of reduced power monoids of finite abelian groups}, 

    {https://arxiv.org/abs/2510.17533}.

    \bibitem{Re26a}
A.~Reinhart, \emph{On the system of length sets of power monoids},
  {https://arxiv.org/abs/2508.10209}.

    \bibitem{Sa12}
    A. Sárközy, \emph{On additive decompositions of the set of quadratic residues modulo} $p$, Acta Arith. \textbf{155} (2012), No. 1, 41-51.
    
    \bibitem{Sh67}
    J. Shafer, \emph{Note on power semigroups}, Math. Japon. \textbf{12} (1967), 32.
    
    \bibitem{Sh-Ta67}
    J. Shafer and T. Tamura, \emph{Power semigroups}, Math. Japon. \textbf{12} (1967), 25--32.
    
    \bibitem{Tr22}
    S. Tringali, \emph{An abstract factorization theorem and some applications}, J. Algebra \textbf{602} (2022), 352--380
    
    \bibitem{Tr25}
    S. Tringali, ``On the isomorphism problem for power semigroups", in: M. Brešar, A. Geroldinger, B. Olberding, and D. Smertnig (eds.), \emph{Recent Progress in Ring and Factorization Theory}, Springer Proc. Math. Stat. \textbf{477}, Springer, 2025.

    \bibitem{Tr-We25}
    S. Tringali and K. Wen, \emph{The automorphism group of the finitary power monoid
    of the integers under addition}, https://arxiv.org/abs/2504.12566.

    \bibitem{Tr-Ya25}
    S. Tringali and W. Yan, \emph{On power monoids and their automorphisms}, J. Combin. Theory Ser. A \textbf{209} (2025), 105961, 16pp.
    \bibitem{2Tr-Ya25}
    S. Tringali and W. Yan, \emph{A conjecture by Bienvenu and Geroldinger on power monoids}, Proc. Amer. Math. Soc. \textbf{153} (2025), 913--919.

    \bibitem{3Tr-Ya25}
    S. Tringali and W. Yan, \emph{Torsion groups and the Bienvenu-Geroldinger Conjecture}, https://arxiv.org/abs/2601.19592.
     
      
      \end{thebibliography}
\end{document}